\documentclass[12pt,a4paper]{article}

\usepackage{graphicx}%
\usepackage{amsmath,amsthm,amsfonts,amssymb,tikz-cd} 
\usepackage{txfonts,eucal,bm} 
\usepackage{color}

\DeclareMathOperator\Char{Char}%
\DeclareMathOperator\Cl{Cl}%
\DeclareMathOperator\id{id}%
\DeclareMathOperator\Lip{Lip}%

\DeclareMathOperator\GL{GL}%

\DeclareMathOperator\GO{GO}%
\DeclareMathOperator\I{I}%
\DeclareMathOperator\Orth{O}%

\DeclareMathOperator\PO{PO}%

\newcommand{\x}{\times} 

\newcommand{\bPV}{\bP(\vV)} 
\newcommand{\bPVQ}{\bP(\vV,Q)} 

\newcommand{\kl}[1]{^{\leq #1}}
\newcommand{\ClVQ}{\Cl(\vV,Q)}
\newcommand{\VQ}{(\vV,Q)}
\newcommand{\VQtilde}{(\tilde\vV,\tilde{Q})}
\newcommand{\Oweak}{\Orth'}%
\newcommand{\Iweak}{\I'}%
\newcommand{\POweak}{\PO'}%

\renewcommand{\phi}{\varphi}

\renewcommand{\kappa}{\varkappa}

\newcommand{\cF}{{\mathcal F}}
\newcommand{\cG}{{\mathcal G}}
\newcommand{\cH}{{\mathcal H}}

\newcommand{\cM}{{\mathcal M}}

\newcommand{\cS}{{\mathcal S}}

\newcommand{\bC}{{\mathbb C}}

\newcommand{\bN}{{\mathbb N}}
\newcommand{\bP}{{\mathbb P}}
\newcommand{\bR}{{\mathbb R}}

\newcommand{\bZ}{{\mathbb Z}}

\newcommand{\vH}{{\bm H}}
\newcommand{\vL}{{\bm L}}

\newcommand{\vS}{{\bm S}}
\newcommand{\vT}{{\bm T}}

\newcommand{\vV}{{\bm V}}
\newcommand{\vX}{{\bm X}}
\newcommand{\va}{{\bm a}}
\newcommand{\vb}{{\bm b}}

\newcommand{\ve}{{\bm e}}

\newcommand{\vg}{{\bm g}}
\newcommand{\vh}{{\bm h}}
\newcommand{\vi}{{\bm i}}

\newcommand{\vm}{{\bm m}}
\newcommand{\vn}{{\bm n}}
\newcommand{\vp}{{\bm p}}
\newcommand{\vq}{{\bm q}}
\newcommand{\vr}{{\bm r}}
\newcommand{\vs}{{\bm s}}
\newcommand{\vt}{{\bm t}}

\newcommand{\vx}{{\bm x}}
\newcommand{\vy}{{\bm y}}
\newcommand{\vz}{{\bm z}}

\newtheorem{thm}{Theorem}[section]

\newtheorem{cor}[thm]{Corollary}
\newtheorem{lem}[thm]{Lemma}
\theoremstyle{definition}

\theoremstyle{remark}
\newtheorem{rem}[thm]{Remark}
\newtheorem{exa}[thm]{Example}

\begin{document}

\author{Hans Havlicek}
\title{Projective metric geometry\\ and Clifford algebras%
\thanks{Added December 2021: Owing to typographical errors by the author, the conditions on ``$\dim\vV$''
appearing in the published version, formulas (2) and (3), fail to match with
the cited sources. In the present text these mistakes have been remedied.}}
\date{}

\maketitle

\begin{center}
\textit{Dedicated to the memory of Heinrich Wefelscheid}
\end{center}

\begin{abstract}
Each vector space that is endowed with a quadratic form determines its Clifford
algebra. This algebra, in turn, contains a distinguished group, known as the
Lipschitz group. We show that only a quotient of this group remains meaningful
in the context of projective metric geometry. This quotient of the Lipschitz
group can be viewed as a point set in the projective space on the Clifford
algebra and, under certain restrictions, leads to an algebraic description of
so-called kinematic mappings.
\end{abstract}

\section{Introduction}\label{se:intro}

By a metric vector space we mean a finite dimensional vector space $\vV$ (over
a field $F$ of arbitrary characteristic) that is endowed with a quadratic form
$Q$. The description of orthogonal transformations of a metric vector space
$\VQ$ in terms of its associated Clifford algebra $\ClVQ$ has a long history.
We follow the exposition by E.~M.~Schr\"{o}der~\cite{schroe-95a} and provide in
Section~\ref{se:metric-vs} basic facts about a metric vector space $\VQ$ and
its weak orthogonal group $\Oweak\VQ$, which in most cases is generated by
reflections. In Section~\ref{se:clifford}, we collect from various sources
those results about the Clifford algebra $\ClVQ$ which are needed later on.
Section~\ref{se:lipschitz} is based on the work of J.~Helmstetter as summarised
in \cite{helm-12a}. We recall from there the Lipschitz monoid $\Lip\VQ$ and the
twisted adjoint representation of the Lipschitz group $\Lip^\times\VQ$, which
provides a surjective homomorphism onto the weak orthogonal group $\Oweak\VQ$.
\par
The main goal of the present note is the interpretation of the Lipschitz group
$\Lip^\times\VQ$ in projective terms, that is, we consider the projective
metric space $\bP\VQ$ and the projective space on the associated Clifford
algebra $\ClVQ$. Thereby one is immediately facing the following problem: if
the quadratic form $Q$ is replaced by a non-zero multiple, say $cQ$ with $c\in
F\setminus\{0\}$, then this does not affect the geometry of $\bP\VQ$, but the
Clifford algebras $\ClVQ$ and $\Cl(\vV,cQ)$ need not be isomorphic. Therefore,
the usage of Clifford algebras in projective metric geometry at a first sight
appears to be problematic.
\par
We start Section~\ref{se:proj-metric} by introducing in $\bP\bigl(\ClVQ\bigr)$
point sets $\cM\VQ$ and $\cG\VQ$ that arise from a quotient of the Lipschitz
monoid $\Lip\VQ$ and a quotient of the Lipschitz group $\Lip^\times\VQ$. The
latter set can be made into a group in a natural way and as such it acts on the
initial projective metric space $\bP\VQ$. In Theorems \ref{thm:iso.0},
\ref{thm:iso.1} and \ref{thm:iso.2} we carry out a detailed study of this group
action and its kernel, thereby extending previous work of C.~Gunn
\cite{gunn-11a}, \cite{gunn-17a}, R.~Jurk \cite{jurk-81a}, M.~Hagemann and
D.~Klawitter \cite{klaw+h-13a}, \cite{klaw-15a}, E.~M.~Schr\"{o}der
\cite{schroe-87b} and others. Since the details are somewhat involved, an
alternative point of view is adopted in Tables
\ref{tab:theta}--\ref{tab:theta_mod}. These tables allow us to read off all
those instances, where a kind of ``kinematic mapping'' for the projective weak
orthogonal group $\POweak\VQ$ can be obtained. Next, in Section~\ref{se:comp},
we return to the problem sketched above by comparing the Clifford algebras
$\ClVQ$ and $\Cl(\vV,cQ)$. From a result by M.-A.~Knus
\cite[Ch.~IV,~(7.1.1)]{knus-91a}, we are in a position to identify the
underlying vector spaces of these algebras in such a way that, firstly, their
even subalgebras $\Cl_0\VQ$ and $\Cl_0(\vV,cQ)$ coincide (as algebras),
secondly, their odd parts $\Cl_1\VQ$ and $\Cl_1(\vV,cQ)$ are the same (as
vector subspaces), thirdly, the two multiplications are related in a manageable
way. Using this identification, it turns out that all our results from
Section~\ref{se:proj-metric} remain unaltered when going over from $Q$ to $cQ$.
This is due to the fact that the Clifford algebra $\ClVQ$ just serves as a kind
of ``container'' for its even and odd part, but we never use any element of
$\ClVQ$ from outside these two subspaces. Finally, Section~\ref{se:future}
provides a list of open questions that may lead to future research.
\par
Let us close by pointing out that our note is not intended to be a critical
survey. We therefore mainly quote such work that will clear the way to previous
contributions. Also, whenever we just refer to other sources without using
them, we usually do not emphasise diverging definitions, differing hypotheses
and other deviations from our approach.

\section{Metric vector spaces}\label{se:metric-vs}

Let $\vV$ be a vector space with finite dimension $n+1\geq 0$ over a
(commutative) field $F$. We suppose that $\vV$ is equipped with a quadratic
form $Q\colon \vV\to F$; the zero form is not excluded. There is a widespread
literature about quadratic forms; see, for example, \cite[Ch.~8]{cohn-03a},
\cite{elma+k+m-08a}, \cite{lam-05a} or \cite{tayl-92a}. We adopt the
terminology from \cite{schroe-95a} by addressing $\VQ$ as a \emph{metric vector
space}. A non-zero vector $\va \in \vV$ is called \emph{regular} if $Q(\va)\neq
0$ and \emph{singular} otherwise. Observe that none of these attributes applies
to the zero vector. A subspace of $\vV$ is said to be \emph{totally singular}
if all its non-zero vectors are singular.
\par
Let $B\colon \vV \x \vV \to F\colon (\vx ,\vy )\mapsto Q(\vx +\vy )-Q(\vx
)-Q(\vy )$ denote the \emph{polar form} of $Q$, which is a symmetric bilinear
form. Then, for all $\vx \in \vV$, $B(\vx ,\vx )= 2 Q(\vx )$.
\emph{Orthogonality} is written as $\perp$; that is, for all $\vx,\vy\in\vV$,
$\vx\perp\vy$ means $B(\vx,\vy)=0$. Each subset $\vS\subseteq \vV$ determines
the subspace $\vS^\perp :=\{\vx\in \vV\mid \vx\perp\vy \mbox{~for
all~}\vy\in\vS\}$ of $\vV$. In particular, $\vV^\perp$ is called the
\emph{radical} of $B$. The form $B$ is said to be \emph{non-degenerate}
provided that $\vV^\perp=\{0\}$.
\par
Let $\VQtilde$ also be a metric vector space over $F$. A linear bijection
$\psi\colon\vV\to\tilde\vV$ is called a \emph{similarity} if $c Q=\tilde Q\circ
\psi$ for some $c\in F^\times:=F\setminus\{0\}$. Provided that $Q(\vV)\neq\{
0\}$, the scalar $c$ is uniquely determined by $\psi$ and it will be addressed
as the \emph{ratio} of $\psi$. Whenever $Q(\vV)=\{0\}$ we adopt the convention
to consider only $1\in F^\times$ as being the \emph{ratio} of $\psi$. An
\emph{isometry} is understood to be a similarity of ratio $c=1$.
\par
We recall that the \emph{general orthogonal group} $\GO\VQ$ is that subgroup of
the general linear group $\GL(\vV)$ which comprises all similarities of $\VQ$
onto itself. All isometries of $\VQ$ onto itself constitute the
\emph{orthogonal group} $\Orth\VQ$. The \emph{weak orthogonal
group}\footnote{We follow here the terminology and notation from
\cite{elle-77a}. In German this group is known under the name ``ortho\-gonale
Gruppe im engeren Sinne''. $\Oweak\VQ$ must not be confused with the derived
group of $\Orth\VQ$; see \cite[p.~39]{cohn-03a}.} $\Oweak\VQ$ consists of all
isometries of $\VQ$ that fix the radical $\vV^\perp$ elementwise. Each regular
vector $\vr\in\vV$ determines the mapping
\begin{equation}\label{eq:reflect}
    \xi_{\vr} \colon \vV\to \vV\colon
    \vx \mapsto \vx - {B( \vr,\vx)}{Q(\vr)}^{-1} \vr .
\end{equation}
We call $\xi_\vr$ the \emph{reflection} of $\VQ$ in the direction of $\vr$ and
note that $\xi_\vr\in\Oweak\VQ$. Under $\xi_{\vr}$ all vectors in $\vr^{\perp}$
are fixed and $\vr$ goes over to $-\vr$. Also, $\xi_{\vr}$ is the identity on
$\vV$ if, and only if, $\vr$ is a regular vector in the radical $\vV^\perp$;
this can only happen in case of characteristic $2$; see
\cite[1.6.2]{schroe-95a}. We are now in a position to write up a version of the
classical Cartan-Dieudonn\'{e} Theorem as follows. Each isometry $\phi\in\Oweak\VQ$
is a product of reflections, except when $F$ and $\VQ$ satisfy one of the
subsequent conditions \eqref{eq:ex1} or \eqref{eq:ex2} for some basis
$\{\ve_0,\ve_1,\ldots,\ve_n\}$ of $\vV$ and all $\vx=\sum_{j=0}^{n}x_j\ve_j$
with $x_j\in F$:
\begin{align}
    |F|=2,\; \dim\vV > 2       &\mbox{~~and~~} Q(\vx) = x_0x_1 ;        \label{eq:ex1}\\
    |F|=2,\; \dim\vV\geq 4     &\mbox{~~and~~} Q(\vx) = x_0x_1+x_2x_3 . \label{eq:ex2}
\end{align}
We refer to \cite{goet-68b}, \cite{goet-68a}, \cite{knes-70a} for proofs and to
\cite{brady+m-15a}, \cite{elle-77a}, \cite{elle-78a}, \cite{elle+f-88a},
\cite{elle+f+n-84a}, \cite{frank+m-87a}, \cite{helm-17c}, \cite{klop-77a},
\cite[p.~18]{lam-05a}, \cite{nolte-76a}, \cite[1.6.3]{schroe-95a},
\cite[pp.~156--159]{tayl-92a} for further details, generalisations and
additional references. Let us just mention that the reflections of $\VQ$
generate a \emph{proper subgroup} of its weak orthogonal group whenever $F$ and
$\VQ$ meet the requirements of \eqref{eq:ex1} or \eqref{eq:ex2}.

\section{Clifford algebras}\label{se:clifford}

Let $\VQ$ be a metric vector space over $F$ (as in Section~\ref{se:metric-vs})
and let $\ClVQ$ denote its \emph{Clifford algebra}; see, among others,
\cite[pp.~35ff. and 101ff.]{chev-97a}, \cite[8.4]{cohn-03a},
\cite[Ch.~8,~Ch.~13]{grove-02a}, \cite[Ch.~3]{helm+m-08a}
\cite[Ch.~IV]{knus-91a} or \cite[Ch.~5]{lam-05a}. In our study we shall adopt
two widely used conventions. Firstly, we identify $1\in F$ with the unit
element of the $F$-algebra $\ClVQ$ and, secondly, we consider $\vV$ as being a
subspace of $\ClVQ$. So $\ClVQ$ is the universal associative and unital algebra
over $F$ that is generated by $\vV$ and subject to the relations $Q(\vx)=\vx^2$
for all $\vx\in\vV$. Consequently, for all $\vx,\vy\in\vV$, we have
$B(\vx,\vy)=\vx\vy+\vy\vx$. We now write up some well known properties of
$\ClVQ$ in order to fix our notation.
\par
The Clifford algebra $\ClVQ$ is $\bZ/(2\bZ)$-graded and so it is the direct sum
of the \emph{even part} $\Cl_0\VQ$, which is a subalgebra of $\ClVQ$, and the
\emph{odd part} $\Cl_1\VQ$. If $\vh\in\Cl_i\VQ$, $i\in\{0,1\}$, then we say
that $\vh$ is \emph{homogeneous} of \emph{degree} $i$ and write
$\partial\vh=i$. Given any subset $\vS\subseteq\ClVQ$ we let
$\vS_i:=\vS\cap\Cl_i\VQ$ for $i\in\{0,1\}$ and we denote by $\vS^\times$ the
set of all units (w.r.t.\ multiplication) in $\vS$.
\par
The \emph{main involution} $\sigma\colon\ClVQ\to\ClVQ$ is the only algebra
endomorphism of $\ClVQ$ such that $\vx\mapsto -\vx$ for all $\vx\in\vV$. Under
$\sigma$ all elements of $\Cl_0\VQ$ remain fixed, any $\vh\in\Cl_1\VQ$ goes
over to $-\vh$. The \emph{reversal} $\alpha\colon\ClVQ\to\ClVQ$ is the only
algebra antiendomorphism of $\ClVQ$ such that $\vx\mapsto \vx$ for all
$\vx\in\vV$. Each of the mappings $\sigma$ and $\alpha$ is a bijection leaving
invariant $\Cl_0\VQ$ and $\Cl_1\VQ$.
\par
$\ClVQ$ is endowed with an (increasing) \emph{canonical filtration} by
subspaces $\Cl\kl{k}\VQ$, $k\in\bZ$, as follows
\cite[pp.~108--109]{helm+m-08a}: if $k<0$ then $\Cl\kl{k}\VQ=\{0\}$; if $k\geq
0$ then $\Cl\kl{k}\VQ$ is that subspace of $\ClVQ$ which is generated by all
products of at most $k$ vectors from $\vV$. Thereby an empty product of vectors
is understood to be $1\in F\subseteq\ClVQ$. If $\{\ve_0,\ve_1,\ldots,\ve_n\}$
is a basis of $\vV$, then we obtain a basis of $\ClVQ$ as
\begin{equation}\label{eq:Cl-basis}
    \bigl\{\ve_{j_1}\ve_{j_2}\cdots \ve_{j_k} \mid 0\leq j_1<j_2<\cdots<j_k\leq n \bigr\} .
\end{equation}
\par
Let $\vT$ be a subspace of $\vV$. The restriction $Q|\vT$ makes $\vT$ into a
metric vector space. The unital subalgebra of $\ClVQ$ generated by $\vT$ will
be considered as the Clifford algebra of $(\vT,Q|\vT)$.
\par
Suppose that $Q(\vV)=\{0\}$. Then $\ClVQ$ coincides (upon writing ``$\wedge$''
instead of ``$\cdot$'') with the exterior algebra $\bigwedge\vV$. Let $\vm$ be
any element of $\Cl_1\VQ=\bigoplus_{j\in\{1,3,5, \ldots\}}\bigwedge^j\vV$. From
$\bigwedge\vV$ being $\bN$-graded (see \cite[p.~185]{helm+m-08a}), the product
of $\vm$ and any $\vn\in\ClVQ=\bigwedge\vV$ belongs to the subspace
$\bigoplus_{k\in\{1,2,3, \ldots\}}\bigwedge^k\vV$. Therefore $\vm$ fails to be
invertible and we note, for later use:
\begin{equation}\label{eq:Q=0}
    Q(\vV)=\{0\} \mbox{~~implies~~} \Cl_1^\times\VQ=\emptyset .
\end{equation}
\par
The following results are taken from \cite[Ch.~IV,~(7.1.1)]{knus-91a} in a form
tailored to our needs. Let $\VQ$ and $\VQtilde$ denote metric vector spaces and
suppose that $\psi\colon\vV\to\tilde\vV$ is a similarity of ratio $c\in
F^\times$. Then there is a unique homomorphism of algebras $\Cl_0(\psi)\colon
\Cl_0\VQ\to \Cl_0\VQtilde$ such that, for all $\vx,\vy\in\vV$,
\begin{equation}\label{eq:Cl0psi}
    \Cl_0(\psi)(\vx\vy) = c^{-1}\psi(\vx)\psi(\vy) .
\end{equation}
Furthermore, there is a unique linear mapping $\Cl_1(\psi)\colon \Cl_1\VQ\to
\Cl_1\VQtilde$ such that, for all $\vp\in\Cl_0\VQ$ and all $\vx\in\vV$,
\begin{equation}\label{eq:Cl1psi}
    \begin{aligned}
    \Cl_1(\psi)(\vp\vx) &= \Cl_0(\psi)(\vp)\cdot\psi(\vx), \\
    \Cl_1(\psi)(\vx\vp) &= \psi(\vx)\cdot\Cl_0(\psi)(\vp) .
    \end{aligned}
\end{equation}
Take notice that, for all $\vx\in\vV$, $\Cl_1(\psi)(\vx)=\psi(\vx)$ follows
from \eqref{eq:Cl1psi} by letting $\vp=1$. This motivates our name
\emph{Clifford extension} of $\psi$ for the mapping
\begin{equation}\label{eq:Clpsi}
    \Cl_0(\psi)\oplus\Cl_1(\psi) = : \Cl(\psi)\colon \ClVQ\to\Cl\VQtilde .
\end{equation}
Since $\psi^{-1}$ is a similarity of ratio $c^{-1}$, there exists
$\Cl(\psi^{-1})=\Cl(\psi)^{-1}$. Thus, by virtue of $\Cl_0(\psi)$, the even
Clifford algebras $\Cl_0\VQ$ and $\Cl_0\VQtilde$ are isomorphic.
\par
Even though the domain of $\Cl(\psi)$ is the entire Clifford algebra $\ClVQ$,
we shall predominantly apply this mapping to homogeneous elements of $\ClVQ$.
In particular, the following formula will turn out crucial, as it describes to
which extent $\Cl(\psi)$ ``deviates'' from an isomorphism of algebras. Given
\emph{homogeneous} elements $\vm,\vn\in\ClVQ$ we assert that
\begin{equation}\label{eq:Clpsi-zwei}
    \Cl(\psi)(\vm \vn)
    = c^{-\partial\vm\partial\vn}\Cl(\psi)(\vm)\cdot\Cl(\psi)(\vn) .
\end{equation}
By the additivity of $\Cl(\psi)$ and the law of distributivity, it suffices to
verify \eqref{eq:Clpsi-zwei} when $\vm=\va_1\va_2\cdots\va_r$ and
$\vn=\va_{r+1}\va_{r+2}\cdots\va_{r+s}$ with
$\va_1,\va_2,\ldots,\va_{r+s}\in\vV$ and $r,s\geq 0$. There are four cases:
\par
\emph{Case} 1: $r$ and $s$ are even. Here \eqref{eq:Clpsi-zwei} holds
trivially, since $\Cl_0(\psi)$ is a homomorphism of algebras.
\par
\emph{Case} 2: $r$ is even and $s$ is odd. By the first equation in
\eqref{eq:Cl1psi} and Case~1,
\begin{equation*}\label{}
    \begin{aligned}
        \Cl_1(\psi)(\vm\vn)
        &=\Cl_1(\psi)\bigl(\vm(\va_{r+1}\va_{r+2}\cdots\va_{r+s-1})\va_{r+s}\bigr) \\
        &=\Cl_0(\psi)(\vm\va_{r+1}\va_{r+2}\cdots\va_{r+s-1})\cdot\psi(\va_{r+s}) \\
        &=\Cl_0(\psi)(\vm)\cdot\Cl_0(\psi)(\va_{r+1}\va_{r+2}\cdots\va_{r+s-1})\cdot\psi(\va_{r+s}) \\
        &=\Cl_0(\psi)(\vm)\cdot\Cl_1(\psi)(\vn) .
    \end{aligned}
\end{equation*}
\par
\emph{Case} 3: $r$ is odd and $s$ is even. Writing
$\vm=\va_1(\va_2\cdots\va_r)$ allows us to proceed in analogy to the previous
case, thereby using the second equation in \eqref{eq:Cl1psi}.
\par
\emph{Case} 4: $r$ and $s$ are odd. Now $\vm=(\va_1\va_2\cdots\va_{r-1})\va_r$
and $\vn=\va_{r+1}(\va_{r+2}\cdots\va_{r+s})$ together with \eqref{eq:Cl0psi}
and \eqref{eq:Cl1psi} establishes \eqref{eq:Clpsi-zwei}.
\par
Next, let $\vm_1,\vm_2,\ldots,\vm_k$, $k\geq 0$, be homogeneous elements of
$\ClVQ$ such that precisely $p$ of them are of degree $1$. Then there is a
unique integer $q\geq 0$ with $2q\leq p\leq 2q+1$. From \eqref{eq:Clpsi-zwei},
we therefore obtain
\begin{equation}\label{eq:Clpsi-viele}
    \Cl(\psi)(\vm_1\vm_2\cdots\vm_k)
    =c^{-q}\Cl(\psi)(\vm_1)\cdot\Cl(\psi)(\vm_2)\cdots\Cl(\psi)(\vm_k) .
\end{equation}
\par
There are two immediate consequences of \eqref{eq:Clpsi-viele}. Given a
homogeneous element $\vm\in\ClVQ$ we have
\begin{equation}\label{eq:Clpsi-alpha}
    \bigl(\Cl(\psi)\circ\alpha\bigr)(\vm)
    =\bigl(\tilde\alpha\circ\Cl(\psi)\bigr)(\vm),
\end{equation}
where $\tilde\alpha$ denotes the reversal on $\Cl\VQtilde$. If, moreover, $\vm$
is invertible, then
\begin{equation}\label{eq:Clpsi-inv}
    \Cl(\psi)(\vm)\cdot\Cl(\psi)(\vm^{-1})=c^{\partial\vm} ,
\end{equation}
which in turn shows that $\Cl(\psi)(\vm)$ is invertible.

\section{Lipschitz groups}\label{se:lipschitz}

The following exposition runs along the lines of the survey \cite{helm-12a} and
the summary in \cite[5.10]{helm+m-08a}. Historical remarks and additional
results may be retrieved from \cite{helm-13a}, \cite{helm-14a},
\cite{helm-17b}, \cite{helm-18a} and \cite[pp.~220--230]{loun-01b}. According
to \cite[Def.~2.1]{helm-12a} the \emph{Lipschitz monoid} $\Lip\VQ$ is the
multiplicative monoid in $\ClVQ$ generated by the union of $F$, $\vV$ and the
set
\begin{equation}\label{eq:erzeuger}
    \bigl\{ 1+\vs\vt \mid \vs,\vt\in\vV,\; Q(\vs)=Q(\vt)=B(\vs,\vt)=0 \bigr\}.
\end{equation}
The Lipschitz monoid $\Lip\VQ$ is already generated by $\vV$ except when one of
the following applies: (i) $Q(\vV)=\{0\}$; (ii) $F$ and $\VQ$ satisfy
\eqref{eq:ex1}; (iii) $F$ and $\VQ$ satisfy \eqref{eq:ex2}; see
\cite[(7)~Thm.]{helm-05a}, \cite[Thm.~2.2]{helm-12a}.
\par
Three noteworthy properties of any $\vm\in\Lip\VQ$ are as follows
\cite[(2)~Thm.]{helm-05a}, \cite[Thm.~2.7]{helm-12a}: (i)
$\alpha(\vm)\in\Lip\VQ$; (ii) $\vm\alpha(\vm)=\alpha(\vm)\vm\in F$; (iii)
$\vm\alpha(\vm)\neq 0$ characterises $\vm$ as being invertible. Therefore all
invertible elements of $\Lip\VQ$ constitute a group, the so-called
\emph{Lipschitz group} $\Lip^\times\VQ$. Furthermore, for all $k\geq 0$ and all
$\vz\in\Cl\kl{k}\VQ$, we have $\vm\vz\alpha(\vm)\in\Cl\kl{k}\VQ$; see
\cite[(23)~Cor.]{helm-05a}, \cite[Thm.~2.8]{helm-12a}. This implies, for any
$\vm\in\Lip^\times\VQ$ and all $\vx\in\vV$, that $\vm\vx\sigma(\vm)^{-1}\in\vV$
and that $Q\bigl(\vm\vx\sigma(\vm)^{-1}\bigr)=Q(\vx)$. The mapping
\begin{equation}\label{eq:twist}
    \xi\colon \Lip^\times\VQ\to\Oweak\VQ\colon
    \vp\mapsto \bigl(\xi_\vp\colon \vx\mapsto \vp\vx\sigma(\vp)^{-1}\bigr)
\end{equation}
is a surjective homomorphism of groups \cite[(35)~Thm.]{helm-05a},
\cite[Thm.~3.2]{helm-12a}; we follow \cite{atiy+b+s-64a} by addressing $\xi$ as
the \emph{twisted adjoint representation} of $\Lip^\times\VQ$. For any regular
vector $\vr\in\vV$ we clearly have $\vr\in\Lip^\times\VQ$ and the above
definition reproduces the reflection $\xi_\vr$ as in \eqref{eq:reflect}.

\begin{rem}
Any element $1+\vs\vt$ appearing in \eqref{eq:erzeuger} is in the Lipschitz
group $\Lip^\times\VQ$, since $(1+\vs\vt)\alpha(1+\vs\vt)=1$. An easy
calculation gives, for all $\vx\in\vV$,
\begin{equation*}
    \xi_{1+\vs\vt}(\vx)= \vx + B(\vt,\vx)\vs - B(\vs,\vx)\vt .
\end{equation*}
If $\vs,\vt$ are linearly dependent, then $\vs\vt=0$ and so
$\xi_{1+\vs\vt}=\id_\vV$. Otherwise, $\xi_{1+\vs\vt}$ fixes precisely the
vectors of the subspace $\{\vs,\vt\}^\perp$, which has codimension $\leq 2$ in
$\vV$.
\end{rem}

In order to describe the kernel of the twisted adjoint representation
\eqref{eq:twist}, we recall the definition of the \emph{graded centre} of
$\ClVQ$. It is defined as
\begin{equation*}
    Z^g\bigl(\ClVQ\bigr):=Z^g_0\bigl(\ClVQ\bigr)\oplus Z^g_1\bigl(\ClVQ\bigr) ,
\end{equation*}
where $Z^g_i\bigl(\ClVQ\bigr)$, $i\in\{0,1\}$, comprises precisely those
$\vp\in\Cl_i\VQ$ which satisfy $\vp\vq=(-1)^{\partial\vp\partial\vq}\vq\vp$ for
all homogeneous $\vq\in\ClVQ$; see \cite[(3.5.2)]{helm+m-08a} or
\cite[p.~152]{knus-91a}. By \eqref{eq:twist}, for all $\vp\in\Lip^\times\VQ$
and all vectors $\vx\in\vV$, we have
\begin{equation*}
     \xi_\vp(\vx)
    =\vp\vx\sigma(\vp)^{-1}
    =(-1)^{\partial\vp}\vp\vx\vp^{-1}
    =(-1)^{\partial\vp\partial\vx}\vp\vx\vp^{-1} .
\end{equation*}
Therefore, using that $\vV$ generates $\ClVQ$ as an algebra, we readily arrive
at the intermediate result
\begin{equation*}
    F^\times\subseteq\ker\xi = \Lip^\times\VQ \cap Z^g\bigl(\ClVQ\bigr) .
\end{equation*}
From \cite[(5.8.7)~Lemma]{helm+m-08a}, the graded centre of $\ClVQ$ equals the
subalgebra generated by $\vV^\perp$, which in turn may be viewed as
$\Cl(\vV^\perp,Q|\vV^\perp)$. We therefore have
\begin{equation}\label{eq:ker-xi}
    F^\times\subseteq\ker\xi = \Lip^\times\VQ \cap \Cl(\vV^\perp,Q|\vV^\perp).
\end{equation}
\par
The above description of the graded centre $Z^g\bigl(\ClVQ\bigr)$ as the
subalgebra of $\ClVQ$ generated by $\vV^\perp$ can also be read off from
\cite[(1.8)~a),~(1.9)~a)]{jurk-81a}. Likewise, the result about $\ker\xi$ may
easily be derived from \cite[(22)~Cor.]{helm-05a} or
\cite[(2.2)~Satz]{jurk-81a}. However, the author of \cite{jurk-81a} states some
essential results without proof (just quoting his thesis). The corresponding
proofs, despite their announcement in \cite{jurk-81a}, never got published.
\par
Below we collect a few more results, which are to be used later on.
\begin{lem}\label{lem:ker-xi}
Let $\VQ$ be a metric vector space. Then the kernel of the twisted adjoint
representation $\xi$ of the Lipschitz group $\Lip^\times\VQ$ satisfies the
following properties.
\begin{enumerate}
\item\label{lem:ker-xi.a} If $Q(\vV^\perp)=\{0\}$, then\/ $\ker
    \xi=\ker_0\xi$ and so\/ $\ker\xi$ is a subgroup of\/
    $\Lip_0^\times\VQ$.

\item\label{lem:ker-xi.b} Given any regular vector $\vr\in\vV^\perp$ we
    have $\vr\cdot(\ker_0\xi) = \ker_1\xi$. Therefore, whenever
    $Q(\vV^\perp)\neq\{0\}$, $\ker\xi$ is not a subgroup of\/
    $\Lip_0^\times\VQ$.

\item\label{lem:ker-xi.c} If $\dim\vV^\perp\leq 1$, then\/
    $\ker_0\xi=F^\times$.

\item\label{lem:ker-xi.d} If $\dim\vV^\perp\geq 2$, then for any
    two-dimensional subspace $\vL\subseteq\vV^\perp$ there are linearly
    independent vectors $\va,\vb\in\vL$ such that
    \begin{equation}\label{eq:ker-xi-T}
        \bigl\{x+y\va\vb \mid x\in F^\times,\;y\in F \bigr\}
        \subseteq \ker_0\xi  .
    \end{equation}
\end{enumerate}
\end{lem}

\begin{proof}
\eqref{lem:ker-xi.a} This is an immediate consequence of \eqref{eq:Q=0},
applied to $\Cl(\vV^\perp,Q|\vV^\perp)$, and \eqref{eq:ker-xi}.
\par
\eqref{lem:ker-xi.b} By \eqref{eq:ker-xi}, $\vr\in\vV^\perp$ implies
$\vr\in\ker\xi$, which proves the assertions.
\par
\eqref{lem:ker-xi.c} We infer from $\dim\vV^\perp\leq 1$ that
$\Cl_0(\vV^\perp,Q|\vV^\perp)=F$. Hence \eqref{eq:ker-xi} gives
$F^\times\subseteq \ker_0\xi\subseteq \Cl_0^\times(\vV^\perp,Q|\vV^\perp) =
F^\times$.
\par
\eqref{lem:ker-xi.d} We distinguish three cases: (i) $\vL$ is totally singular;
(ii) $\vL$ contains no singular vectors; (iii) $\vL$ contains a regular and a
singular vector. In the first two cases we choose any linearly independent
vectors $\va,\vb\in \vL$. If (iii) applies, we choose $\va\in\vL$ regular and
$\vb\in\vL$ singular. Now pick any element $x+y\va\vb$ as appearing in
\eqref{eq:ker-xi-T}. In view of \eqref{eq:ker-xi}, it suffices to verify that
$x+y\va\vb\in\Lip_0^\times\VQ$. In case (i), this follows from
\begin{equation*}
    x+y\va\vb=x\bigl(1+(x^{-1}y\va)\vb\bigr),
\end{equation*}
$x\in F^\times$ and $Q(x^{-1}y\va)=Q(\vb)=B(x^{-1}y\va,\vb)=0$. Otherwise,
$\va$ is regular and so $\Char F=2$. By writing
\begin{equation*}
    x+y\va\vb=\va(x\va^{-1}+y\vb) ,
\end{equation*}
it remains to verify that $x\va^{-1}+y\vb$ is regular. In case (ii), this turns
out trivial. In case (iii), we have $Q(x\va^{-1}+y\vb)=x^2Q(\va^{-1})+0\neq 0$.
\end{proof}

\section{Projective metric geometry}\label{se:proj-metric}

Let $\vV$ be a vector space over $F$ as described at the beginning of
Section~\ref{se:metric-vs}. By the \emph{projective space} $\bPV$ we mean here
the set of all subspaces of $\vV$ with \emph{incidence} being symmetrised
inclusion \cite[p.~30]{buek+c-95a}. The \emph{dimension}\footnote{In order to
avoid ambiguity, we shall frequently add the attribute ``projective'' when
speaking about the dimension of a projective space.} of $\bP(\vV)$ is one less
than the dimension of $\vV$; that is, $\bP(\vV)$ has projective dimension $n$.
We adopt the usual geometric terms: \emph{points}, \emph{lines} and
\emph{planes} are the subspaces of $\vV$ with (vector) dimension one, two, and
three, respectively. Likewise, any subspace $\vT$ of $\vV$ gives rise to a
projective space $\bP(\vT)$, which is a substructure of $\bPV$. The general
linear group $\GL(\vV)$ acts in a canonical way on $\bP(\vV)$: any
$\kappa\in\GL(\vV)$ determines a \emph{projective collineation} on $\bP(\vV)$,
which is given by $\vX\mapsto\kappa(\vX)$ for all $\vX\in\bP(\vV)$. The kernel
of this action of $\GL(\vV)$ equals $F^\times \id_\vV$.
\par
Next, assume $\VQ$ to be a metric vector space. Then $Q$ can been used to
furnish the projective space with ``additional structure'', thus making it into
a \emph{projective metric space} $\bPVQ$. Thereby, for all $c\in F^\times$, the
spaces $\bPVQ$ and $\bP(\vV,cQ)$ are considered as being \emph{equal}. We refer
to \cite{schroe-95a} for a detailed description under the assumption
$Q(\vV)\neq\{0\}$; otherwise any ``additional structure'' arising from $Q$ is
trivial. Let us recall a few notions derived from $\VQ$ that remain unchanged
if $Q$ is replaced by $cQ$. The orthogonality relations of $\VQ$ and of
$(\vV,cQ)$ coincide. All points $F\vs$ with $\vs\in\vV$ being singular
constitute the \emph{absolute quadric} $\cF\VQ$ of $\bPVQ$.\footnote{Since we
allow $Q$ to be the zero form, $\cF\VQ$ may coincide with the point set of
$\bPV$.} This quadric does not alter when going over to $cQ$. Also, we have
$\GO\VQ=\GO(\vV,cQ)$, $\Orth\VQ=\Orth(\vV,cQ)$ and $\Oweak\VQ=\Oweak(\vV,cQ)$.
\par
\par
In contrast, the Clifford algebras $\ClVQ$ and $\Cl(\vV,cQ)$, $c\in F^\times$,
\emph{need not be isomorphic}; see Example~\ref{exa:nichtiso}, where it is also
shown that an analogous statement holds for the associated Lipschitz groups.
Nevertheless, for the remaining part of this section, we shall make extensive
use of the Clifford algebra $\ClVQ$. The problem of how things change when
going over to $\Cl(\vV,cQ)$ will be addressed in Section~\ref{se:comp}.
\par
By the above, any isometry $\phi\in\Oweak\VQ$ determines a projective
collineation of $\bP\VQ$. This action of $\Oweak\VQ$ on $\bP\VQ$ has the kernel
\begin{equation}\label{eq:kern}
    \Iweak\VQ := \Oweak\VQ\cap\{\id_\vV,-\id_\vV\}.
\end{equation}
The quotient of $\Oweak\VQ$ by $\Iweak\VQ$ is the \emph{projective weak
orthogonal group}, in symbols $\POweak\VQ$. Then
\begin{equation}\label{eq:kern=1}
    |{\Iweak\VQ}| = 1 \;\Leftrightarrow\; \Bigl(
     \vV=\{0\}          \mbox{~or~}
     \vV^\perp\neq\{0\} \mbox{~or~}
     \Char F = 2 \Bigr) .
\end{equation}
On the other hand, $|{\Iweak\VQ}|\neq 1$ implies $|{\Iweak\VQ}|=2$.
\par
Now we change over to the projective space on the Clifford algebra $\ClVQ$,
where we introduce several point sets. Given such a set, say $\cS$, we define
$\cS_i$, $i\in\{0,1\}$, to be the subset of $\cS$ comprising all points that
are contained in $\Cl_i\VQ$. We start by defining
\begin{equation}\label{eq:H}
    \cH\VQ := \bigl\{F\vp\mid \vp\in\Cl_0^\times\VQ \cup \Cl_1^\times\VQ \bigr\}
\end{equation}
and proceed by making $\cH\VQ$ into a (multiplicative) group as follows:
$(F\vp)(F\vq):=F(\vp\vq)$ for all $F\vp,F\vq\in\cH\VQ$. Clearly, there is a
canonical isomorphism of groups
\begin{equation}\label{eq:iso-H}
    \bigl(\Cl_0^\times\VQ \cup \Cl_1^\times\VQ\bigr)/F^\times \to \cH\VQ
    \colon F^\times\vp \mapsto F\vp .
\end{equation}
So, essentially, the two groups from above are the same. The Lipschitz monoid
$\Lip\VQ$ gives rise to the point set
\begin{equation}\label{eq:M}
    \cM\VQ:=\bigl\{F\vp\mid 0\neq\vp\in\Lip\VQ\bigr\}  .
\end{equation}
By the definition of $\Lip\VQ$, the point set $\cM\VQ$ is the disjoint union of
$\cM_0\VQ$ and $\cM_1\VQ$.
\begin{rem}
The sets $\cM_0\VQ$ and $\cM_1\VQ$ are algebraic varieties of the projective
spaces on $\Cl_0\VQ$ and $\Cl_1\VQ$, respectively. See \cite[p.~673]{helm-12a}
and \cite{helm-11a}, where a wealth of further properties of these varieties
can be found. In particular, all subspaces of $\bP\bigl(\ClVQ\bigr)$ whose
point set is contained in $\cM\VQ$ have been determined there. Let us just
mention the following particular case. If $\dim\vV\leq 3$, then $\cM_0\VQ$
resp.\ $\cM_1\VQ$ comprises \emph{all points} of $\bP\bigl(\Cl_0\VQ\bigr)$
resp.\ $\bP\bigl(\Cl_1\VQ\bigr)$ \cite[(31)~Lemma]{helm-05a}.
\end{rem}

\begin{rem}
In general, the point set $\cM\VQ$ cannot be made into a monoid by following
the path taken above. This is because the product of two non-zero elements of
$\Lip\VQ$ may be the zero vector, which fails to span a point.
\end{rem}

Our third point set is
\begin{equation}\label{eq:G}
    \cG\VQ:=\cH\VQ\cap\cM\VQ =\bigl\{F\vp\mid \vp\in\Lip^\times\VQ\bigr\} ,
\end{equation}
which is a subgroup of $\cH\VQ$. The canonical isomorphism from
\eqref{eq:iso-H} determines (by restriction) the isomorphism of groups
\begin{equation}\label{eq:iso-G}
    \Lip^\times\VQ / F^\times \to \cG\VQ .
\end{equation}
\par
The Lipschitz group $\Lip^\times\VQ$ contains $F^\times$ as a normal subgroup.
The representation \eqref{eq:twist} factors through the canonical homomorphism
$\Lip^\times\VQ \to \Lip^\times\VQ / F^\times$. We therefore have a surjective
homomorphism of groups
\begin{equation}\label{eq:twist-F}
    \Lip^\times\VQ / F^\times \to\Oweak\VQ \colon F^\times\vp \mapsto \xi_\vp ,
\end{equation}
the \emph{twisted adjoint representation} of the quotient group $\Lip^\times\VQ
/ F^\times$. By virtue of the inverse of \eqref{eq:iso-G}, the twisted adjoint
representation \eqref{eq:twist-F} and the canonical action of $\Oweak\VQ$ on
$\bP\VQ$, the group $\cG\VQ$ as in \eqref{eq:G} acts on the projective space
$\bP\VQ$. Explicitly, for all $F\vp\in\cG\VQ$ and all $\vX\in\bP\VQ$, this
action of $\cG\VQ$ takes the form
\begin{equation}\label{eq:G-action}
    F\vp \mapsto \bigl(\vX\mapsto\xi_\vp(\vX)=\vp\vX\sigma(\vp)^{-1}\bigr) .
\end{equation}
Furthermore, the action of $\cG\VQ$ on $\bP\VQ$ yields a surjective
homomorphism of groups
\begin{equation}\label{eq:theta}
    \theta\colon\cG\VQ\to\POweak\VQ=\Oweak\VQ / \Iweak\VQ
    \colon F\vp \mapsto \Iweak\VQ\circ \xi_\vp ;
\end{equation}
see \eqref{eq:kern}. By our construction, $\ker\theta$ is just the kernel of
the group action described in \eqref{eq:G-action}. This means
\begin{equation}\label{eq:kern_theta}
    \ker\theta =\bigl\{F\vp\in\cG\VQ \mid \xi_\vp\in \Iweak\VQ \bigr\} .
\end{equation}

\begin{rem}\label{rem:compat}
Let any $F\vm\in\cH\VQ$ be given. The (group theoretic) \emph{left translation}
by $F\vm$, that is the mapping $F\vq\mapsto (F\vm)(F\vq)$ for all
$F\vq\in\cH\VQ$, extends to a projective collineation of the ambient projective
space. Obviously, the \emph{left translation}
$\lambda_\vm\in\GL\bigl(\ClVQ\bigr)$, which is given by $\vz\mapsto\vm\vz$ for
all $\vz\in\ClVQ$, provides a solution. The same properties hold, \emph{mutatis
mutandis}, for the \emph{right translation} by $F\vm$ and its counterpart
$\rho_\vm\colon \vz\mapsto\vz\vm$ on the Clifford algebra.
\par
Given any $F\vm\in\cG\VQ$ the above observations clearly remain true when
replacing $\cH\VQ$ with $\cG\VQ$. However, $\cG\VQ$ satisfies yet another
property, which appears to be more substantial. The mapping that sends any
$F\vq\in\cG\VQ$ to its inverse $(F\vq)^{-1}$ also extends to a projective
collineation of the ambient projective space. Such a collineation is determined
by the reversal $\alpha$, since $\vq\alpha(\vq)\in F^\times$ for all
$\vq\in\Lip^\times\VQ$; see the noteworthy properties (i)--(iii) of the
Lipschitz monoid mentioned at the beginning of Section~\ref{se:lipschitz}.
\end{rem}

We proceed by examining in detail the kernel of the surjective homomorphism
$\theta$ appearing in \eqref{eq:theta}. We also investigate whether or not the
subgroup $\theta\bigl(\cG_0\VQ\bigr)$ coincides with its ambient group
$\POweak\VQ$. Clearly, whenever $\cG_0\VQ=\cG\VQ$, the answer to the latter
question is affirmative. The large number of cases makes us split our findings
into three theorems, according to the dimension of the radical $\vV^\perp$.

\begin{thm}\label{thm:iso.0}
Let $\VQ$ be a metric vector space with $\dim\vV^\perp=0$. Then the surjective
homomorphism $\theta\colon\cG\VQ\to\POweak\VQ$ has the following properties.
\begin{enumerate}
\item\label{thm:iso.0.a} If\/ $\dim\vV=0$, then $\ker\theta=\{F1\}$ and\/
    $\cG_0\VQ=\cG\VQ$.

\item\label{thm:iso.0.b} If\/ $\dim\vV>0$ and\/ $\Char F=2$, then
    $\ker\theta=\{F1\}$ and $\theta\bigl(\cG_0\VQ\bigr)\neq\POweak\VQ$.

\item\label{thm:iso.0.c} If\/ $\dim\vV$ is odd and $\Char F\neq 2$, then\/
    $\ker\theta$ comprises precisely two points, namely $F1\in\cG_0\VQ$ and
    one more point in $\cG_1\VQ$. Furthermore,
    $\theta\bigl(\cG_0\VQ\bigr)=\POweak\VQ$.

\item\label{thm:iso.0.d} If\/ $\dim\vV>0$ is even and\/ $\Char F\neq 2$,
    then\/ $\ker\theta$ comprises precisely two points, both of which in
    $\cG_0\VQ$. Furthermore, $\theta\bigl(\cG_0\VQ\bigr)\neq\POweak\VQ$.
\end{enumerate}
\end{thm}

\begin{proof}
To begin with, we note that
\begin{equation}\label{eq:iso.0.lokal}
    \dim\vV^\perp=0
    \mbox{~~implies~~}
    \ker\xi=\ker_0\xi=F^\times,
\end{equation}
as follows readily from Lemma~\ref{lem:ker-xi}~\eqref{lem:ker-xi.a} and
\eqref{lem:ker-xi.c}.
\par
\eqref{thm:iso.0.a} The assertions hold, since $\cG\VQ=\cG_0\VQ=\{F1\}$.
\par
\eqref{thm:iso.0.b} Due to $\Char F=2$ and \eqref{eq:kern=1}, we have
$\Iweak\VQ=\{\id_\vV\}$. Therefore \eqref{eq:kern_theta} and
\eqref{eq:iso.0.lokal} give $\ker\theta=\{F1\}$. As $\vV^\perp=\{0\}$ and
$\dim\vV>0$, there exists a regular vector $\vr\in\vV$ and so
$F\vr\in\cG\VQ\setminus\cG_0\VQ$. Since $\ker\theta=\{F1\}$ means that $\theta$
is injective, we obtain $\theta(F\vr)\in \POweak\VQ \setminus
\theta\bigl(\cG_0\VQ\bigr)$.
\par
\eqref{thm:iso.0.c} Now \eqref{eq:kern=1} implies $|{\Iweak\VQ}|=2$. There
exists an \emph{orthogonal} basis $\{\ve_0,\ve_1,\ldots,\ve_n\}$ of $\VQ$ and
we put
\begin{equation}\label{eq:e}
    \ve:=\ve_0\ve_1\cdots\ve_n .
\end{equation}
From $\vV^\perp=\{0\}$ we obtain $\ve\in\Lip^\times\VQ$. Together with
$\dim\vV\geq 1$ and $\Char F\neq 2$ this shows $\xi_\ve=-\id_\vV\neq\id_\vV$.
So, from \eqref{eq:kern_theta} and \eqref{eq:iso.0.lokal},
$\ker\theta=\{F1,F\ve\}$ is a group of order two. Clearly, $F1\in\ker_0\theta$.
\par
As $\dim\vV$ is odd, $F\ve\in\ker_1\theta$. We have
\begin{equation}\label{eq:theta-G0=G1}
      \theta\bigl(\cG_0\VQ\bigr)
    = \theta\bigl(F\ve\cdot\cG_0\VQ\bigr)
    = \theta\bigl(\cG_1\VQ\bigr),
\end{equation}
whence $\theta\bigl(\cG_0\VQ\bigr) = \theta\bigl(\cG\VQ\bigr)=\POweak\VQ$.
\par
\eqref{thm:iso.0.d} We may repeat the reasoning from \eqref{thm:iso.0.c} up to
the end of the first paragraph. By contrast, now $\dim\vV>0$ is even, whence
$F\ve\in\ker_0\theta$. In analogy with \eqref{thm:iso.0.b}, there is a regular
vector $\vr\in\vV$ and so $F\vr\in\cG\VQ \setminus\cG_0\VQ$. Taking into
account that $\ker\theta$ is contained in $\cG_0\VQ$, we obtain
$\theta(F\vr)\in \POweak\VQ \setminus \theta\bigl(\cG_0\VQ\bigr)$.
\end{proof}
Note that under the hypotheses of Theorem~\ref{thm:iso.0}~\eqref{thm:iso.0.b}
the bilinear form $B$ is non-degenerate and alternating. Therefore, $\dim\vV$
is necessarily even.


\begin{thm}\label{thm:iso.1}
Let $\VQ$ be a metric vector space with $\dim\vV^\perp=1$. Then the surjective
homomorphism $\theta\colon\cG\VQ\to\POweak\VQ$ has the following properties.
\begin{enumerate}
\item\label{thm:iso.1.a} If $Q(\vV^\perp)=\{0\}$, then $\ker\theta=\{F1\}$.

\item\label{thm:iso.1.b} If $Q(\vV^\perp)=\{0\}$ and $\dim\vV=1$, then
    $\cG_0\VQ=\cG\VQ$.

\item\label{thm:iso.1.c} If $Q(\vV^\perp)=\{0\}$ and $\dim\vV> 1$, then
    $\theta\bigl(\cG_0\VQ\bigr) \neq \POweak\VQ$.

\item\label{thm:iso.1.d} If $Q(\vV^\perp)\neq\{0\}$, then $\ker\theta$
    comprises precisely two points, namely $F1\in\cG_0\VQ$ and one more
    point in $\cG_1\VQ$. Furthermore, $\theta\bigl(\cG_0\VQ\bigr) =
    \POweak\VQ$.
\end{enumerate}
\end{thm}

\begin{proof}
First of all, from \eqref{eq:kern=1}, we have that
\begin{equation}\label{eq:iso.1.lokal}
    \dim\vV^\perp=1
    \mbox{~~implies~~}
    \Iweak\VQ=\{\id_\vV\} .
\end{equation}
\par
\eqref{thm:iso.1.a} From Lemma~\ref{lem:ker-xi}~\eqref{lem:ker-xi.a} and
\eqref{lem:ker-xi.c}, we have $\ker\xi=\ker_0\xi=F^\times$. Thus, using
\eqref{eq:kern_theta} and \eqref{eq:iso.1.lokal}, we arrive at
$\ker\theta=\{F1\}$.
\par
\eqref{thm:iso.1.b} By our hypotheses, we have $Q(\vV)=Q(\vV^\perp)=\{0\}$. So
\eqref{eq:Q=0} yields that $\Cl_1\VQ$ contains no regular vector. Therefore
$\cG_0\VQ=\cG\VQ$, as required.
\par
\eqref{thm:iso.1.c} Due to $\vV^\perp\neq\vV$, there exists a regular vector in
$\vr\in\vV$ and so $F\vr\in\cG\VQ \setminus\cG_0\VQ$. From \eqref{thm:iso.1.a},
$\theta$ is injective, whence $\theta(F\vr)\in \POweak\VQ \setminus
\theta\bigl(\cG_0\VQ\bigr)$.
\par
\eqref{thm:iso.1.d} There is a regular vector $\vr\in\vV^\perp$. Using
Lemma~\ref{lem:ker-xi}~\eqref{lem:ker-xi.b} and \eqref{lem:ker-xi.c}, we
obtain\footnote{We note that $F1\oplus F\vr$ is a subalgebra of
$\Cl(\vV^\perp,Q|\vV^\perp)$ and as such an inseparable quadratic extension
field of $F$.} $\ker\xi = F^\times \cup F^\times\vr$. Together with
\eqref{eq:kern_theta} and \eqref{eq:iso.1.lokal} this implies
$\ker\theta=\{F1,F\vr\}$ and $F\vr\cdot\cG_0\VQ=\cG_1\VQ$. Thus
$\theta\bigl(\cG_0\VQ\bigr)=\theta\bigl(\cG\VQ\bigr)=\POweak\VQ$ follows (by
replacing $\ve$ with $\vr$) in analogy to \eqref{eq:theta-G0=G1} .
\end{proof}

Regarding Theorem~\ref{thm:iso.1}~\eqref{thm:iso.1.d}, it seems worth pointing
out that $Q(\vV)\neq\{0\}$ implies $\Char F=2$. Together with $\dim\vV^\perp=1$
this forces $\dim\vV$ to be odd, since $B$ induces a non-degenerate alternating
bilinear form on the quotient space $\vV/\vV^\perp$.

\begin{thm}\label{thm:iso.2}
Let $\VQ$ be a metric vector space with $\dim\vV^\perp\geq 2$. Then the
surjective homomorphism $\theta\colon\cG\VQ\to\POweak\VQ$ has the following
properties.
\begin{enumerate}
\item\label{thm:iso.2.a} If $Q(\vV^\perp)=\{0\}$, then $\ker\theta$
    contains at least\/ $|F|$ points from\/ $\cG_0\VQ$ but no points from\/
    $\cG_1\VQ$.
\item\label{thm:iso.2.b} If $Q(\vV^\perp)=\{0\}$ and\/
    $\dim\vV=\dim\vV^\perp$, then\/ $\cG_0\VQ=\cG\VQ$.
\item\label{thm:iso.2.c} If $Q(\vV^\perp)=\{0\}$ and\/
    $\dim\vV>\dim\vV^\perp$, then
    $\theta\bigl(\cG_0\VQ\bigr)\neq\POweak\VQ$.
\item\label{thm:iso.2.d} If $Q(\vV^\perp)\neq\{0\}$, then\/ $\ker\theta$
    contains at least\/ $|F|$ points from\/ $\cG_0\VQ$ and at least\/ $|F|$
    points from\/ $\cG_1\VQ$. Furthermore, $\theta\bigl(\cG_0\VQ\bigr) =
    \POweak\VQ$.
\end{enumerate}
\end{thm}

\begin{proof}
\eqref{thm:iso.2.a} Lemma~\ref{lem:ker-xi}~\eqref{lem:ker-xi.a} gives
$\ker\xi=\ker_0\xi$ and, from \eqref{eq:kern=1}, we have
$\Iweak\VQ=\{\id_\vV\}$. Thus \eqref{eq:kern_theta} shows
$\ker\theta=\ker_0\theta\subseteq\cG_0\VQ$. Let $\vL$ be any two-dimensional
subspace of $\vV^\perp$. By adopting the terminology from
Lemma~\ref{lem:ker-xi}~\eqref{lem:ker-xi.d} and by substituting $x:=1$ in
\eqref{eq:ker-xi-T}, we arrive at
\begin{equation}\label{eq:punkte}
    \{F(1+y\va\vb)\mid y\in F\}\subseteq\ker_0\theta
    \mbox{~~and~~}
    \bigl|\{F(1+y\va\vb)\mid y\in F\}\bigr| = |F| .
\end{equation}
\par
\eqref{thm:iso.2.b} The assertion follows from \eqref{eq:Q=0}.
\par
\eqref{thm:iso.2.c} From $\dim\vV>\dim\vV^\perp$, there exists a regular vector
$\vr\in\vV$ and so $F\vr\in\cG\VQ \setminus \cG_0\VQ$. In view of
\eqref{thm:iso.2.a}, $\ker\theta$ is contained in $\cG_0\VQ$. This in turn
establishes $\theta(F\vr)\in\POweak\VQ\setminus\theta\bigl(\cG_0\VQ\bigr)$.
\par
\eqref{thm:iso.2.d} There exists a two-dimensional subspace $\vL$ of
$\vV^\perp$ that contains a regular vector $\va$, say. We pick any vector
$\vb\in\vL$ such that $\va,\vb$ are linearly independent. According to the
proof of Lemma~\ref{lem:ker-xi}~\eqref{lem:ker-xi.d} we now use these vectors
to obtain \eqref{eq:ker-xi-T} and, as in \eqref{thm:iso.2.a}, we substitute
there $x:=1$. In this way we get a point set as in \eqref{eq:punkte}. This
implies that $\ker_1\theta$ contains at least $|F|$ points, namely all points
of the form $F\bigl(\va+yQ(\va)\vb\bigr)$ with $y$ varying in $F$.
\end{proof}

Our description of $\ker\theta$ in Theorems~\ref{thm:iso.0}, \ref{thm:iso.1}
and \ref{thm:iso.2} improves \cite[(2.3)~Satz]{jurk-81a} in two ways: The
result b) from there describes an analogue of our surjective homomorphism
$\theta$ onto the group $\POweak\VQ$; however, it is based upon a subgroup of
$\cH\VQ$ that in general is \emph{larger} than our $\cG\VQ$. The result c) from
there coincides with our findings whenever $\Lip^\times\VQ$ is generated by all
regular vectors of $\VQ$, but provides no information about the exceptional
cases \eqref{eq:ex1} and \eqref{eq:ex2}.
\par
Clearly, the surjective homomorphism $\theta$ as in \eqref{eq:theta} turns out
to be an \emph{isomorphism of $\cG\VQ$ onto $\POweak\VQ$} if, and only if,
$\ker\theta$ contains no point other than $F1$. There are few possibilities for
this to happen. All of them can be read off from Table~\ref{tab:theta}. The
first entry in each row (other than the header) provides a reference to the
corresponding theorem, the remaining entries summarise the conditions that have
to be met. Entries in braces are redundant and could be omitted.
\par
Likewise, there is a rather small number of instances such that
$\theta|\cG_0\VQ$ establishes an \emph{isomorphism of $\cG_0\VQ$ onto
$\POweak\VQ$}. An exhaustive summary is given in Table~\ref{tab:theta_0}. Note
that there is a single overlap between Table~\ref{tab:theta} and
Table~\ref{tab:theta_0}. It pertains the trivial case $\dim\vV=0$, where
$\cG\VQ=\cG_0\VQ$.
\par
There is one more noteworthy situation, where $\theta$ fails to be injective,
but $\ker\theta$ is a \emph{group of order two}; the details are displayed in
Table~\ref{tab:theta_mod}. Here the group $\cG_0\VQ$ is equipped with the
distinguished point $F\ve$, which does not depend on the choice of the
orthogonal basis $\{\ve_0,\ve_1,\ldots,\ve_n\}$ of $\vV$ that has been used in
\eqref{eq:e} when defining $\ve$. The left translation $\lambda_\ve$ (right
translation $\rho_\ve$) acts on $\bP\bigl(\ClVQ\bigr)$ as a projective
collineation; see Remark~\ref{rem:compat}. It is easy to verify that, for all
$j\in\{0,1,\ldots,n\}$, we have $\ve\ve_j=-\ve_j\ve$. Using the basis
\eqref{eq:Cl-basis} of $\ClVQ$, we therefore obtain
\begin{equation*}
    \lambda_\ve|\Cl_0\VQ =  \rho_\ve|\Cl_0\VQ
    \mbox{~~and~~}
    \lambda_\ve|\Cl_1\VQ =  -\rho_\ve|\Cl_1\VQ .
\end{equation*}
Thus, even though $\lambda_\ve$ and $\rho_\ve$ act differently on
$\bP\bigl(\ClVQ\bigr)$, their actions on $\bP\bigl(\Cl_0\VQ\bigr)$ and
$\bP\bigl(\Cl_1\VQ\bigr)$ coincide.
\begin{table}[ht!]
\centering
\begin{tabular}{c|c|c|c|c}
    \hline
    Theorem                            &$\dim\vV^\perp$&$Q(\vV^\perp)$&$\dim\vV$      &$\Char F$\\ \hline
    \ref{thm:iso.0}~\eqref{thm:iso.0.a}&($=0$)         &($=\{0\}$)    &$=0$           &         \\
    \ref{thm:iso.0}~\eqref{thm:iso.0.b}&$=0$           &($=\{0\}$)    &$>0$ (and even)&$=2$     \\
    \ref{thm:iso.1}~\eqref{thm:iso.1.a}&$=1$           &$=\{0\}$      &               &         \\

\end{tabular}
\caption{$\cG\VQ \cong \POweak\VQ$~~(via $\theta$)}\label{tab:theta}
\end{table}
\begin{table}[ht!]
\centering
\begin{tabular}{c|c|c|c|c}
    \hline
    Theorem                            &$\dim\vV^\perp$&$Q(\vV^\perp)$&$\dim\vV$&$\Char F$\\ \hline
    \ref{thm:iso.0}~\eqref{thm:iso.0.a}&($=0$)         &($=\{0\}$)    &$=0$     &         \\
    \ref{thm:iso.0}~\eqref{thm:iso.0.c}&$=0$           &($=\{0\}$)    &odd      &$\neq2$  \\
    \ref{thm:iso.1}~\eqref{thm:iso.1.b}&$=1$           &$=\{0\}$      &$=1$     &         \\
    \ref{thm:iso.1}~\eqref{thm:iso.1.d}&$=1$           &$\neq\{0\}$   &(odd)    &($=2$)
\end{tabular}
\caption{$\cG_0\VQ\cong\POweak\VQ$~~(via $\theta|\cG_0\VQ$)}\label{tab:theta_0}
\end{table}
\begin{table}[ht!]
\centering
\begin{tabular}{c|c|c|c|c}
    \hline
    Theorem                            &$\dim\vV^\perp$&$Q(\vV^\perp)$&$\dim\vV$    &$\Char F$\\ \hline
    \ref{thm:iso.0}~\eqref{thm:iso.0.d}&$=0$           &($=\{0\}$)    &$>0$ and even&$\neq2$
\end{tabular}
\caption{$\cG\VQ/\{F1,F\ve\}\cong\POweak\VQ$}\label{tab:theta_mod}
\end{table}
\par
To conclude this section, let us point out the following. If one of the
situations from Table~\ref{tab:theta} occurs, then we may consider
$\theta^{-1}$ as being a bijective \emph{``kinematic mapping''} for the group
$\POweak\VQ$. Note that this just a name for a series of examples rather than a
general definition of such a mapping. Also, if one of the situations from
Table~\ref{tab:theta_0} occurs, we have a bijective \emph{``kinematic
mapping''} for the group $\POweak\VQ$ given by
$\bigl(\theta|\cG_0\VQ\bigr)^{-1}$. Under the restrictions of
Table~\ref{tab:theta_mod} we still have a kind of \emph{``kinematic mapping''},
but here one element of $\POweak\VQ$ is represented by an \emph{unordered pair
of points} from $\cG\VQ$. Some of the examples in \cite[3.4]{klaw-15a} and
\cite[Sect.~6]{klaw+h-13a} fit into the above concepts. However, the quoted
works should be read with caution due to several misprints.

\section{A comparison of Clifford algebras}\label{se:comp}

We now switch back to a problem that we encountered in
Section~\ref{se:proj-metric}. Given a metric vector space $\VQ$ and a constant
$c\in F^\times$ what is the relationship between the Clifford algebras $\ClVQ$
and $\Cl(\vV,cQ)$? For a closer look, we take into account that the identity
$\id_\vV$ is a similarity of ratio $c$ from $\VQ$ onto $(\vV,cQ)$. (Recall our
convention that $c=1$ whenever $Q(\vV)=\{0\}$.) Therefore, according to
\eqref{eq:Clpsi}, we obtain a linear bijection
\begin{equation}\label{eq:Clid}
    \Cl(\id_\vV)\colon\ClVQ\to\Cl(\vV,cQ) .
\end{equation}
This linear bijection allows us pulling back the algebra structure from
$\Cl(\vV,cQ)$ to $\ClVQ$, which amounts to introducing a ``new'' multiplication
$\odot_c$ on the vector space $\ClVQ$. The algebra obtained in this way is
isomorphic to $\Cl(\vV,cQ)$ and will be abbreviated as $\Cl(\vV,Q,\odot_c)$. A
bridge between the initial and the new multiplication is provided by
\eqref{eq:Cl0psi} and \eqref{eq:Cl1psi}. We read off from there, for all
$\vx,\vy\in\vV$ and all $\vp\in\Cl_0\VQ$:
\begin{equation}\label{eq:c.mult-zwei}
    \vx\vy=c^{-1}\vx\odot_c\vy,\quad \vp\vx=\vp\odot_c\vx,\quad \vx\vp=\vx\odot_c\vp .
\end{equation}
Similarly, one may write up analogues of \eqref{eq:Clpsi-zwei},
\eqref{eq:Clpsi-viele}, \eqref{eq:Clpsi-alpha} and \eqref{eq:Clpsi-inv}. In
what follows right now, we shall adopt a slightly different point of view. We
investigate the Clifford algebras of metric vector spaces $\VQ$ and $\VQtilde$
admitting a similarity $\psi$ of ratio $c\in F^\times$ with $\Cl(\psi)$ playing
the role of the linear bijection \eqref{eq:Clid}. We shall return to
$\Cl(\vV,Q,\odot_c)$ only at the end of this section.

\begin{exa}\label{exa:nichtiso}
Let $\vV$ be a one-dimensional vector space over the field $\bR$ of real
numbers and let $\vi\in\vV$ be non-zero. We define a quadratic form
$Q\colon\vV\to\bR$ by $Q(\vi)=-1$. Then $\ClVQ$ and the field $\bC$ of complex
numbers are isomorphic as $\bR$-algebras, as follows from $\vi^2=-1$.
Furthermore, let $\VQtilde$ be isometric to $(\vV,-Q)$, whence there is a
similarity $\psi\colon\vV\to\tilde\vV$ of ratio $-1$. From $\tilde
Q\bigl(\psi(\vi)\bigr)=\psi(\vi)^2=1$, the Clifford algebra $\Cl\VQtilde$
contains zero divisors\footnote{For example,
$\bigl(1+\psi(\vi)\bigr)\bigl(1-\psi(\vi)\bigr)=0$.}, whence the algebras
$\Cl\VQtilde$ and $\ClVQ\cong\bC$ cannot be isomorphic; see also \cite[Ex.~1.5,
pp.~104--105]{lam-05a} or the table of real Clifford algebras
\cite[p.~123]{lam-05a}.
\par
The Lipschitz group $\Lip^\times\VQ$ reads $\bR^\times\cup \bR^\times\vi$. In
$\Lip^\times\VQ$ we have $\vi^2=-1$, whereas no element of
$\Lip^\times\VQtilde=\bR^\times\cup \bR^\times\psi(\vi)$ squares to $-1$. So
the Lipschitz groups of $\VQ$ and $\VQtilde$ cannot be isomorphic either. In
contrast, the quotient groups $\Lip^\times\VQ/\bR^\times$ and
$\Lip^\times\VQtilde/\bR^\times$ both have order two and so they are
isomorphic; see Theorem~\ref{thm:psi-faktor}.
\end{exa}

\begin{thm}\label{thm:psi}
Let $\psi\colon\vV\to\tilde\vV$ be a similarity of ratio $c\in F^\times$ of
metric vector spaces $\VQ$ and $\VQtilde$. Then the Clifford extension
$\Cl(\psi)$ has the following properties.
\begin{enumerate}
\item\label{thm:psi.a} $\Cl(\psi)$ maps the canonical filtration of\/
    $\ClVQ$ onto the canonical filtration of\/ $\Cl\VQtilde$.
\item\label{thm:psi.b} $\Cl(\psi)$ maps the group\/ $\Cl_0^\times\VQ \cup
    \Cl_1^\times\VQ$ onto\/ $\Cl_0^\times\VQtilde \cup
    \Cl_1^\times\VQtilde$, the Lipschitz monoid\/ $\Lip\VQ$ onto\/
    $\Lip\VQtilde$ and, consequently, the Lipschitz group\/
    $\Lip^\times\VQ$ onto\/ $\Lip^\times\VQtilde$.
\item\label{thm:psi.c} For any $\vm\in\Lip^\times\VQ$, the isometries
    $\xi_\vm\in\Oweak\VQ$ and $\tilde\xi_{\Cl(\psi)(\vm)}\in\Oweak\VQtilde$
    satisfy $ \psi\circ\xi_{\vm} = \tilde\xi_{\Cl(\psi)(\vm)}\circ\psi$.
\end{enumerate}
\end{thm}
\begin{proof}
\eqref{thm:psi.a} Pick any integer $k\geq 0$. According to
\eqref{eq:Clpsi-viele}, $\Cl(\psi)$ takes any product of $k$ vectors from $\vV$
to a product of $k$ vectors from $\tilde\vV$, and an analogous statement holds
for $\Cl(\psi^{-1})=\Cl(\psi)^{-1}$. Thus the image of $\Cl\kl{k}\VQ$ under
$\Cl(\psi)$ equals $\Cl\kl{k}\VQtilde$, as required.
\par
\eqref{thm:psi.b} To begin with, choose any $\vm\in\Cl_0^\times\VQ \cup
\Cl_1^\times\VQ$. From \eqref{eq:Clpsi-inv}, the element $\Cl(\psi)(\vm)$ is in
$\Cl_0^\times\VQtilde \cup \Cl_1^\times\VQtilde$.
\par
We now show that $\Cl(\psi)$ sends any generator of $\Lip\VQ$, that is to mean
any element $\vg$ from $F$, $\vV$ or the set \eqref{eq:erzeuger}, to a
generator of $\Lip\VQtilde$ of the same kind. If $\vg$ is in $F\cup\vV$, then
this is obvious. If $\vg$ belongs to the set \eqref{eq:erzeuger} or, more
explicitly, if $\vg = 1+\vs\vt$ with $\vs,\vt\in\vV$ subject to
$Q(\vs)=Q(\vt)=B(\vs,\vt)=0$, then \eqref{eq:Clpsi-zwei} implies
$\Cl(\psi)(1+\vs\vt) = 1+c^{-1}\psi(\vs)\psi(\vt)$. As $\psi$ is a similarity,
we obtain $\tilde Q\bigl(c^{-1}\psi(\vs)\bigr)=\tilde
Q\bigl(\psi(\vt)\bigr)=\tilde B \bigl(c^{-1}\psi(\vs),\psi(\vt)\bigr)=0$,
whence $\Cl(\psi)(\vg)$ has the required property.
\par
Next, let any $\vm\in\Lip\VQ$ be given. By definition, $\vm$ is a product of
$k\geq 0$ generators $\vg_1,\vg_2,\ldots,\vg_k$ that come from $F$, $\vV$ or
the set \eqref{eq:erzeuger}. From \eqref{eq:Clpsi-viele}, there is an integer
$q\geq 0$ such that
\begin{equation*}
    \Cl(\psi)(\vm)
    =c^{-q}\Cl(\psi)(\vg_1)\cdot\Cl(\psi)(\vg_2)\cdots\Cl(\psi)(\vg_k) .
\end{equation*}
Thus, by the above, $\Cl(\psi)(\vm)\in\Lip\VQtilde$.
\par
Finally, \eqref{thm:psi.b} follows by repeating the above considerations with
the similarity $\psi^{-1}$ instead of $\psi$.
\par
\eqref{thm:psi.c} Choose any $\vx\in\vV$. Let $p$ be the number of factors with
degree $1$ in the product $\vm\vx\sigma(\vm)^{-1}$ and denote by $q$ the
integer satisfying $2q\leq p \leq 2q+1$. As
$\partial\bigl(\sigma(\vm)^{-1}\bigr)=\partial\vm$, we may argue as follows. If
$\partial\vm=0$, then $p=1$ and $q=0$. If $\partial\vm=1$, then $p=3$ and
$q=1$. Therefore, we always have $q=\partial\vm$. Now, from
\eqref{eq:Clpsi-viele}, \eqref{eq:Clpsi-inv} and
$\Cl(\psi)\circ\sigma=\tilde\sigma\circ\Cl(\psi)$, we get
\begin{equation*}
\begin{aligned}
    \psi\big(\xi_\vm(\vx)\big) &= \Cl(\psi)\bigl(\vm\vx\sigma(\vm)^{-1}\bigr) \\
    &= c^{-q}\Cl(\psi)(\vm) \cdot \Cl(\psi)(\vx) \cdot \Cl(\psi)\bigl(\sigma(\vm)^{-1}\bigr)\\
    &= c^{-q+\partial\vm}
       \Cl(\psi)(\vm) \cdot \psi(\vx) \cdot \tilde\sigma\bigl(\Cl(\psi)(\vm)\bigr)^{-1}\\
    &= \tilde\xi_{\Cl(\psi)(\vm)}\bigl(\psi(\vx)\bigr) ,
\end{aligned}
\end{equation*}
which completes the proof.
\end{proof}

\begin{thm}\label{thm:psi-faktor}
Under the hypotheses of Theorem\/ \emph{\ref{thm:psi}} the following hold.
\begin{enumerate}
\item\label{thm:psi-faktor.a} For any homogeneous element $\vm\in\ClVQ$ and
    all subspaces $\vH\subseteq\Cl_0\VQ\cup\Cl_1\VQ$, we have
    \begin{equation}\label{eq:mH}
    \begin{aligned}
        \Cl(\psi)(\vm\vH)    &=\Cl(\psi)(\vm)  \cdot\Cl(\psi)(\vH) ,\\
        \Cl(\psi)(\vH\vm)    &=\Cl(\psi)(\vH)  \cdot\Cl(\psi)(\vm) .\\
    \end{aligned}
    \end{equation}
\item\label{thm:psi-faktor.b} The assignment $ F^\times\vp\mapsto
    F^\times\bigl(\Cl(\psi)(\vp)\bigr) = \Cl(\psi)(F^\times\vp)$ defines an
    isomorphism of groups
    \begin{equation*}
        \bigl(\Cl_0^\times\VQ \cup \Cl_1^\times\VQ\bigr)/F^\times \to
        \bigl(\Cl_0^\times\VQtilde \cup \Cl_1^\times\VQtilde\bigr)/F^\times ,
    \end{equation*}
    an isomorphism of monoids\/ $\Lip\VQ / F^\times \to \Lip\VQtilde /
        F^\times$ and, consequently, also an isomorphism of groups\/
    \begin{equation}\label{eq:psi-faktor-lip}
        \Lip^\times\VQ / F^\times \to \Lip^\times\VQtilde / F^\times .
    \end{equation}

\item\label{thm:psi-faktor.c} The twisted adjoint representations of the
    quotient groups\/ $\Lip^\times\VQ / F^\times$ and\/
    $\Lip^\times\VQtilde /F^\times$ are equivalent by virtue of the
    isomorphism \eqref{eq:psi-faktor-lip} and the given similarity
    $\psi\colon\vV\to\tilde\vV$.
\end{enumerate}
\end{thm}
\begin{proof}
The assertions are immediate from \eqref{eq:Clpsi-zwei}, \eqref{eq:Clpsi-inv}
and Theorem \ref{thm:psi}.
\end{proof}

\begin{rem}\label{rem:M+G-tilde}
Let us briefly sketch how to rephrase
Theorem~\ref{thm:psi-faktor}~\eqref{thm:psi-faktor.b} and
\eqref{thm:psi-faktor.c} in terms of the projective spaces
$\bP\bigl(\ClVQ\bigr)$ and $\bP\bigl(\Cl\VQtilde\bigr)$. The Clifford extension
$\Cl(\psi)$ of the given similarity $\psi$ gives rise to a bijection
$\cM\VQ\to\cM\VQtilde$ and it also yields an isomorphism linking the groups
$\cH\VQ$ and $\cH\VQtilde$. Consequently, it determines an isomorphism of the
groups $\cG\VQ$ and $\cG\VQtilde$ as well as their actions on $\bP\VQ$ and
$\bP\VQtilde$, respectively. Therefore, $\Cl(\psi)$ establishes also an
isomorphism between the kernels of these group actions.
\end{rem}

\begin{rem}\label{rem:quadrat}
We still adhere to the hypotheses of Theorem~\ref{thm:psi}. Moreover, we
require $c$ to be a square in $F$. Upon choosing any square root of $c^{-1}$,
say $\sqrt{c^{-1}}$, the following applies. The mapping
$\omega:=\sqrt{c^{-1}}\,\psi$ is an isometry of $\VQ$ onto $\VQtilde$. By the
universal property of Clifford algebras, $\omega$ extends to a unique
isomorphism of algebras $\ClVQ\to\Cl\VQtilde$, which is easily seen to coincide
with $\Cl(\omega)$. Also, we have $\Cl(\omega)=\Cl_0(\psi)\oplus
\sqrt{c^{-1}}\,\Cl_1(\psi)$, whence the isomorphism $\Cl(\omega)$ allows for
alternative proofs of our previous results.\footnote{If $c$ fails to be a
square in $F$, then this can be carried out by going over to metric vector
spaces over an appropriate quadratic extension of $F$.}
\end{rem}

We now switch back to our earlier point of view. Given $\VQ$ and $c\in
F^\times$ we consider $\Cl(\vV,Q,\odot_c)$ as Clifford algebra of $(\vV,cQ)$
with $\odot_c$ being subject to \eqref{eq:c.mult-zwei}. From \eqref{eq:Clpsi},
the even subalgebras of $\ClVQ$ and $\Cl(\vV,Q,\odot_c)$ coincide (as
algebras), as do their odd parts (as vector spaces). Our quest for going over
to the projective space on $\ClVQ$ comes from an observation resulting from
\eqref{eq:mH}: for all homogeneous elements $\vp,\vq\in\ClVQ$, we have
$F(\vp\vq)=F(\vp\odot_c\vq)$ despite the fact that their products $\vp\vq$ and
$\vp\odot_c\vq$ need not coincide. From \eqref{eq:Clpsi-alpha},
Theorem~\ref{thm:psi}, Theorem~\ref{thm:psi-faktor} and
Remark~\ref{rem:M+G-tilde} we readily obtain:

\begin{cor}\label{cor:M+G}
Let $\VQ$ be a metric vector space and let $c\in F^\times$. The following
notions arising from the Clifford algebra $\ClVQ$ do not alter when going over
to the algebra $\Cl(\vV,Q,\odot_c)$:
\begin{enumerate}
    \item\label{cor:M+G.a} The canonical filtration of\/ $\ClVQ$;
    \item\label{cor:M+G.b} for any homogeneous $\vm\in\Cl^\times\VQ$, the
        canonical action of the left translation $\lambda_\vm$
        \emph{(}right translation $\rho_\vm$\emph{)} on the union of the
        projective spaces\/ $\bP\bigl(\Cl_0\VQ\bigr)$ and\/
        $\bP\bigl(\Cl_1\VQ\bigr)$;
        \item\label{cor:M+G.e} the canonical action of the reversal
            $\alpha$ on the union of the projective spaces\/
            $\bP\bigl(\Cl_0\VQ\bigr)$ and\/ $\bP\bigl(\Cl_1\VQ\bigr)$;
    \item\label{cor:M+G.c} the group\/ $\cH\VQ\cong \bigl(\Cl_0^\times\VQ
        \cup \Cl_1^\times\VQ\bigr)/F^\times$ as in\/ \eqref{eq:H};
    \item\label{cor:M+G.d} the point set\/ $\cM\VQ$ arising from the
        Lipschitz monoid\/ $\Lip\VQ$ according to\/ \eqref{eq:M} and the
        group\/ $\cG\VQ\cong \Lip^\times\VQ / F^\times$ as in\/
        \eqref{eq:G};

    \item\label{cor:M+G.f} the action of the group $\cG\VQ$ on the
        projective space\/ $\bP\VQ$ as in\/ \eqref{eq:G-action}.
\end{enumerate}
\end{cor}

\section{Future research}\label{se:future}

We are of the opinion that a closer look at low-dimensional examples should
prove worthwhile. The first interesting class of examples are projective metric
planes ($\dim\vV=3$), since they appear in the theory of \emph{absolute
planes}; see \cite{bach-73a}, \cite{bach-89a}, \cite[Ch.~III]{karz+k-88a},
\cite[3.4.1]{klaw-15a}, \cite[6.1]{klaw+h-13a}, \cite{klop-85a},
\cite{moln-18a}, \cite[4.6]{schroe-95a} and the references therein; furthermore
also \emph{finite Bolyai-Lobachevsky} planes show up here \cite{korch+s-12a},
\cite{korch+s-14a}. In all these examples, the corresponding even Clifford
algebra is a \emph{quaternion algebra} \cite{gross+l-09a}, \cite{voig-11a}. In
particular, an \emph{elliptic plane} gives rise to a \emph{quaternion division
algebra}. Ultimately, one is lead to the following question: to which extent
does the general theory of \emph{kinematic spaces} (including the theory of
\emph{Clifford parallelism}) overlap with our findings as sketched in
Remark~\ref{rem:compat}. We refer, among others, to \cite{broe-73a},
\cite{havl+p+p-21b}, \cite{karz-73a} and \cite{paso-10a}. Going up one
dimension ($\dim\vV=4)$, one finds $\cM_0\VQ$ and $\cM_1\VQ$ as siblings of the
classical \emph{Study quadric} (see \cite{study-13a}) in a projective space of
dimension $7$; also here there are many results scattered over the literature;
see \cite[p.~463]{helm-11a}, \cite[3.4.2]{klaw-15a}, \cite[6.2]{klaw+h-13a} and
\cite[Ch.~11]{seli-05a}.
\par
Another step, still to be taken in a general context, is the inclusion of
\emph{affine metric geometry}. Over the real numbers this task has been
accomplished quite a while ago and leads to what is called a \emph{homogeneous
model}. Related work can be read off from \cite[3.4.2]{klaw-15a},
\cite[6.2]{klaw+h-13a}, \cite{gunn-17a} and \cite{gunn-17b}. However, the
approach used there relies on the \emph{signature} of a real quadratic form, a
notion which is no longer available over an arbitrary field.
\par
Last, but not least, also the general theory should allow for amplification.
The results in \cite{schroe-73a}, where \emph{points and planes} of a
three-dimensional projective space are used to represent motions of metric
planes, suggest to investigate under which conditions the subspaces $\Cl_0\VQ$
and $\Cl_1\VQ$ of the Clifford algebra $\ClVQ$ can be made into a \emph{dual
pair of vector spaces} in some meaningful way.

\small

\newcommand{\Dbar}{\makebox[0cm][c]{\hspace{-2.5ex}\raisebox{0.25ex}{-}}}\newcommand{\cprime}{$'$}

\noindent
Hans Havlicek\\
Institut f\"{u}r Diskrete Mathematik und Geometrie\\
Technische Universit\"{a}t Wien\\
Wiedner Hauptstra{\ss}e 8--10/104\\
1040 Wien\\
Austria\\
\texttt{havlicek@geometrie.tuwien.ac.at}
\end{document}